\documentclass[11pt]{article}
\usepackage{latexsym,amssymb,amsmath}
\usepackage{amssymb,amsthm,upref,amscd}
\usepackage{color}
\usepackage[T1]{fontenc}
\begin{document}
\baselineskip=15pt

\numberwithin{equation}{section}

\newtheorem{thm}{Theorem}[section]
\newtheorem{lem}[thm]{Lemma}
\newtheorem{cor}[thm]{Corollary}
\newtheorem{Prop}[thm]{Proposition}
\newtheorem{Def}[thm]{Definition}
\newtheorem{Rem}[thm]{Remark}
\newtheorem{Ex}[thm]{Example}
\newtheorem{Claim}[thm]{Claim}

\newcommand{\A}{\mathbb{A}}
\newcommand{\B}{\mathbb{B}}
\newcommand{\C}{\mathbb{C}}
\newcommand{\D}{\mathbb{D}}
\newcommand{\E}{\mathbb{E}}
\newcommand{\F}{\mathbb{F}}
\newcommand{\G}{\mathbb{G}}
\newcommand{\I}{\mathbb{I}}
\newcommand{\J}{\mathbb{J}}
\newcommand{\K}{\mathbb{K}}
\newcommand{\M}{\mathbb{M}}
\newcommand{\N}{\mathbb{N}}
\newcommand{\Q}{\mathbb{Q}}
\newcommand{\R}{\mathbb{R}}
\newcommand{\T}{\mathbb{T}}
\newcommand{\U}{\mathbb{U}}
\newcommand{\V}{\mathbb{V}}
\newcommand{\W}{\mathbb{W}}
\newcommand{\X}{\mathbb{X}}
\newcommand{\Y}{\mathbb{Y}}
\newcommand{\Z}{\mathbb{Z}}
\newcommand\ca{\mathcal{A}}
\newcommand\cb{\mathcal{B}}
\newcommand\cc{\mathcal{C}}
\newcommand\cd{\mathcal{D}}
\newcommand\ce{\mathcal{E}}
\newcommand\cf{\mathcal{F}}
\newcommand\cg{\mathcal{G}}
\newcommand\ch{\mathcal{H}}
\newcommand\ci{\mathcal{I}}
\newcommand\cj{\mathcal{J}}
\newcommand\ck{\mathcal{K}}
\newcommand\cl{\mathcal{L}}
\newcommand\cm{\mathcal{M}}
\newcommand\cn{\mathcal{N}}
\newcommand\co{\mathcal{O}}
\newcommand\cp{\mathcal{P}}
\newcommand\cq{\mathcal{Q}}
\newcommand\rr{\mathcal{R}}
\newcommand\cs{\mathcal{S}}
\newcommand\ct{\mathcal{T}}
\newcommand\cu{\mathcal{U}}
\newcommand\cv{\mathcal{V}}
\newcommand\cw{\mathcal{W}}
\newcommand\cx{\mathcal{X}}
\newcommand\ocd{\overline{\cd}}

\def\c{\centerline}
\def\ov{\overline}
\def\emp {\emptyset}
\def\pa {\partial}
\def\bl{\setminus}
\def\op{\oplus}
\def\sbt{\subset}
\def\un{\underline}
\def\al {\alpha}
\def\bt {\beta}
\def\de {\delta}
\def\Ga {\Gamma}
\def\ga {\gamma}
\def\lm {\lambda}
\def\Lam {\Lambda}
\def\om {\omega}
\def\Om {\Omega}
\def\sa {\sigma}
\def\vr {\varepsilon}
\def\va {\varphi}

\title{\bf Existence of solutions for a nonlocal variational problem in $\R^2$ with exponential critical growth}

\author{Claudianor O. Alves\thanks{Partially supported by CNPq/Brazil
304036/2013-7, coalves@dme.ufcg.edu.br}\\
{\small  Universidade Federal de Campina Grande} \\ {\small Unidade Acad\^emica de Matem\'{a}tica} \\ {\small CEP: 58429-900, Campina Grande - Pb, Brazil}
\\
\\
Minbo Yang\thanks{Supported by NSFC (11101374, 11271331) and CNPq/Brazil 500001/2013-8, mbyang@zjnu.edu.cn}\vspace{2mm}
\\
{\small  Department of Mathematics, Zhejiang Normal University} \\ {\small  Jinhua, Zhejiang, 321004, P. R. China.}}

\date{}
\maketitle

\begin{abstract}
We study the existence of solution for the following class of nonlocal problem,
$$
-\Delta u +V(x)u =\Big( I_\mu\ast F(x,u)\Big)f(x,u)  \quad \mbox{in} \quad \R^2,
$$
where $V$ is a positive periodic potential, $I_\mu=\frac{1}{|x|^\mu}$, $0<\mu<2$ and $F(x,s)$ is the primitive function of $f(x,s)$ in the variable $s$.  In this paper, by assuming that the nonlinearity $f(x,s)$ has an exponential critical growth at infinity, we prove the existence of solutions by using variational methods.
 \vspace{0.3cm}


\vspace{0.3cm}

\end{abstract}

\section{Introduction and main results}

At the last years, many attention have been given to the problem 
$$
\left\{\aligned &-\Delta u +V(x)u =\Big( I_\mu\ast F(x,u)\Big)f(x,u)  \quad \mbox{in} \quad \R^N, \\
&  u\in H^{1}({\R}^N),\\
& u(x)>0\ \ \mbox{for all}\ \ x\in\mathbb{R}^N,
\endaligned\right.\eqno{(P)}
$$
where $0<\mu<N$, $I_\mu=\frac{1}{|x|^\mu}$, $F(x,s)$ is the primitive function of $f(x,s)$ in the variable $s$ and $V,f$ are continuous verifying some conditions. Here 
$ I_\mu\ast F(x,u)$ denotes the convolution between $I_\mu$ and $F(.,u(.))$.

This problem comes from looking for standing waves of the nonlinear nonlocal Schr\"{o}dinger
equation which is known to influence the propagation of electromagnetic
waves in plasmas \cite{BC} and also plays an important
role in the theory of Bose-Einstein condensation \cite{D}.
It is used in the description of the quantum theory of a polaron at rest by S. Pekar in 1954 \cite{P1}
and the modeling of an electron trapped
in its own hole in 1976 in the work of P. Choquard, in a certain approximation to Hartree-Fock theory of one-component
plasma \cite{L1}.

If  $F(x,s)=|s|^{q}$, then we arrive at the Choquard-Pekar equation,
\begin{equation}\label{Nonlocal.S1}
 -\Delta u +V(x)u =\Big(\frac{1}{|x|^{\mu}}\ast|u|^{q}\Big)|u|^{q-2}u  \quad \mbox{in} \quad \R^N.
\end{equation}
In the case $N\geq3$, if $V(x)=1$,  Lieb \cite{L1} proved the existence and uniqueness, up to translations,
of the ground state to equation \eqref{Nonlocal.S1}. Later, in \cite{Ls}, Lions showed the existence of
a sequence of radially symmetric solutions to this equation. Involving the properties of the ground state solutions,
 Ma and Zhao \cite{ML} proved that every positive solution of it is radially symmetric and monotone decreasing about some point, under the assumption that a certain set of real numbers, defined in terms of $N, \mu$ and $q$, is nonempty. Under the
same assumption, Cingolani, Clapp and Secchi \cite{CCS1}  proved the existence and
multiplicity results in the electromagnetic case, and established the regularity and decay behavior at infinity of the ground states of \eqref{Nonlocal.S1}. Moroz and Van Schaftingen \cite{MS1} eliminated this restriction and showed the regularity, positivity
and radial symmetry of the ground states for the optimal range of parameters, and
derived decay property at infinity as well. When $V$ is a continuous periodic function with $\inf_{\mathbb{R}^{N}} V(x)> 0$, noticing that the nonlocal term is invariant under translation, one can obtain the existence result easily by applying the Mountain Pass Theorem, see \cite{AC} for example.
For periodic potential $V$ that changes sign and $0$ lies in the gap of the spectrum of the Schr\"{o}dinger operator $-\Delta +V$, the problem is strongly indefinite it have been considered in \cite{BJS}. In that paper, the existence of nontrivial solution with $\mu=1$ and $F(u)=u^2$ have been obtained by using the reduction methods. For a general class of response function $Q$ and nonlinearity $f$, Ackermann \cite{AC} proposed an approach to prove the
existence of infinitely many geometrically distinct weak solutions.

In the study made in the above papers, it was crucial the following Hardy-Littlewood-Sobolev inequality.

\begin{Prop} \cite{LL}$\,\,[Hardy-Littlewood-Sobolev \ inequality]$:\\
Let $t, r>1$ and $0<\mu<N$ with $1/t+\mu/N+1/r=2$. If $f\in
L^t(\R^N)$ and $h\in L^r(\R^N)$, then there exists a sharp constant
$C(t,\mu,r)$, independent of $f,h$, such that
$$
\int_{\R^N}\int_{\R^N}\frac{f(x)h(y)}{|x-y|^\mu}\leq
C(t,\mu,r) |f|_t|h|_r.
$$
\end{Prop}

The above inequality permits to use variational method to get a solution for problem $(P)$, for a large class of nonlinearity $f$, which has in general a subcritical growth. However, we can observe that the Hardy-Littlewood-Sobolev inequality also holds for $N=2$, motivated by this fact, at least from a mathematical point of view, it seems to be interesting to ask if the existence of solution still holds for nonlinearities $f$ having an "{\it exponential subcritical growth}" or "{\it exponential critical growth}" in $\R^2$, since a lot of estimates made for the case $N \geq 3$ cannot be repeated easily for the case $N=2$, when the nonlinearity $f$ has an exponential growth, because in dimension 2, it is well known that the Trudinger-Moser inequality is a crucial tool to work with this type of nonlinearity. Here, we focus our attention for more difficulty case, that is, the "{\it exponential critical growth}" in $\R^2$.   However, we would point out that we cannot say that problem $(P)$ in $\mathbb{R}^{2}$ is a  
{\it nonlinear Choquard equation}, because in dimension 2 the kernel associated with a Choquard equation, namely the therm $I_\mu$,  must involve a logarithmic convolution potential, which does not occur in our problem.  

Since we intend to work with nonlinearity with "{\it exponential critical growth}" in $\R^2$, we mean that the function $f(x,s)$ has
an \emph{exponential critical growth} when it behaves like $e^{\alpha s^2}$ as  $|s|\to+\infty$. More exactly, there exists $\alpha_0>0$ such that
\begin{equation}\label{ecg}
\lim_{|s|\to +\infty}\frac{|f(x,s)|}{e^{\alpha s^2}}=0,\ \ \forall \alpha>\alpha_0,\ \ \hbox{and} \ \ \lim_{|s|\to +\infty}\frac{|f(x,s)|}{e^{\alpha s^2}}=+\infty,\ \ \forall \alpha<\alpha_0.
\end{equation}
The above notation of criticality was introduced by Admurth and Yadava \cite{AY}, see also de Figueiredo, Miyagaki and Ruf \cite{DMR}.

To work with problems where the nonlinearity has an exponential critical growth, one of the most important tools is the Trudinger-Moser inequality, which says that if $\Omega$ is a bounded domain in $\mathbb{R}^2$,  then for all $\alpha>0$ and $u\in
H_0^1(\Omega)$, $e^{\alpha u^2}\in L^1(\Omega)$. Moreover, there exists a positive constant
$C$ such that
\[
\sup_{u\in H_0^1(\Omega) \; : \; \|\nabla u\|_2\leq
1}\int_{\Omega} e^{\beta
 u^2}\leq C|\Omega| \quad\text{if }\alpha \leq 4\pi,
\]
where $|\Omega|$ denotes the Lebesgue measure of $\Omega$. This
inequality is optimal, in the sense that for any growth $e^{\alpha
u^2}$ with $\alpha > 4\pi$ the correspondent supremum is infinite.  In the present paper, we are working in whole $\mathbb{R}^{2}$, this way, it is more convenient for us to use the following Trudinger-Moser type inequality in $H^{1}(\mathbb{R}^2)$ due to Cao \cite{Cao}, which is crucial for our variational arguments
\begin{lem} \label{Trudinger-Moser}
If $\alpha >0$ and $u\in H^{1}(\mathbb{R}^2)$, then
\begin{equation}\label{TM1}
\int_{\mathbb{R}^2}\Big[ e^{ \alpha | u | ^2}
-1 \Big] <\infty .
\end{equation}
Moreover, if $\|\nabla u\| _{L^2(\mathbb{R}^2)}^2\leq 1$,
 $\| u\| _{L^2(\mathbb{R}^2)}\leq M<\infty $ and $\alpha <\alpha
_0=4\pi$, then there exists a constant
$C$, which depends only on $M$ and $\alpha $,
such that
\begin{equation}\label{TM2}
\int_{\mathbb{R}^2}\big[ e^{\alpha | u| ^{2}}
-1 \big] \leq C(M,\alpha ).
\end{equation}
\end{lem}

We assume that  $V:\R^2\to\R$ is continuous and satisfies: \\
$(V)$  \,\, There are $\alpha>0$ and a continuous 1-periodic continuous $V_0:\mathbb{R}^{2} \to \mathbb{R}$  such that
$$
0<\alpha \leq V(x)\leq V_0(x)  \,\,\, \forall x \in \mathbb{R}^{2} \leqno{(1)}
$$
and
$$
|V(x)-V_0(x)| \to 0 \,\,\, |x| \to +\infty. \leqno{(2)}
$$

In the sequel, $E$ denotes the Sobolev space $H^{1}(\R^2)$ equipped with the norm
 $$
 \|u\|:=\left(\int_{\R^2}(|\nabla u|^2+ V( x)|u|^2)\right)^{1/2}
 $$
and $L^s(\R^2)$, for $1 \leq s \leq \infty$, denotes the Lebesgue space endowed with the usual norm $|\,\,\,\,\,\,\,|_s$.

Since the imbedding $H^{1}(\R^2)\hookrightarrow L^p(\R^2)$ is continuous for any $p\in (2,+\infty)$, from the Hardy-Littlewood-Sobolev inequality, there is a best constant $S_p$ verifying
$$
\displaystyle S_p=\displaystyle\inf_{u\in E, u\neq0}\frac{\displaystyle\left(\int_{\R^2}(|\nabla u|^2+ |V|_{\infty}|u|^2)\right)^{1/2}}{\displaystyle\left(\int_{\R^2}\Big( I_\mu\ast |u|^p\Big)|u|^p\right)^{\frac{1}{2p}}}.
$$
Moreover, a standard minimizing argument shows that there exists a positive radial function $u_p\in E$ such that $S_p$ is achieved by $u_p$.

Related to function $f$, we assume that there is a 1-periodic continuous function $\tilde{f}(x,s)$ such that:\\
$$
0\leq \tilde{f}(x,s)\leq {f}(x,s)  \leq \,\,\,  Ce^{4 \pi s^{2}}  \,\,\,\, \forall s \geq 0. \eqno{(f_1)}
$$
There holds
$$
\lim_{s\to 0}\frac{f(x,s)}{s^{\frac{2-\mu}{2}}}=0, \ \ \lim_{s\to 0}\frac{\tilde{f}(x,s)}{s^{\frac{2-\mu}{2}}}=0. \eqno{(f_2)}
$$
There exist ${\theta}\geq\tilde{\theta}>2$, such that
$$
0<\theta F(x,s)\leq 2f(x,s)s,\ \ 0<\tilde{\theta}\tilde{F}(x,s)\leq 2\tilde{f}(x,s)s ~~~~ \forall s>0, \eqno{(f_3)}
$$
where $F(x,t)=\int^t_0f(x,s)ds$ and $\tilde{F}(x,s)=\int^t_0\tilde{f}(x,s)ds$.\\
There exists $p>\frac{4-\mu}{2}$, such that
$$
F(x,s)\geq C_ps^p\,\,\,\, \forall s\geq0 \eqno{(f_4)}
$$
where
$$
C_p>\frac{[\frac{4\theta(p-1)}{(2-\mu)(\theta-2)}]^{\frac{p-1}{2}}S^{p}_p}{p^{\frac{p}{2}}}.
$$
For any fixed $x\in \R^2$, the functions
$$
s\to f(x,s),\ \  \tilde{ f}(x,s)\ \ \hbox{are increasing}. \eqno{(f_5)}
$$
Moreover, $ \tilde{F}(x,s)>F(x,s)$ for any $s\neq0$ and there exists $A \in L^{\infty}(\mathbb{R}^{2})$ verifying
$$
A(x)\to 0 \,\,\, \mbox{as} \,\,\, |x|\to \infty
$$
and
$$
|\tilde{ f}(x,s)-f(x,s)|\leq A(x)\Big(s^{\frac{2-\mu}{2}}+ e^{4\pi s^2} \Big) \,\,\, \forall x \in \mathbb{R}^{2} \,\,\, \mbox{and} \,\,\, \forall s \in \mathbb{R}.\eqno{(f_6)}
$$
The first result of this paper is associated with the periodic case, and it has the following statement
\begin{thm}\label{MRS1}
Assume $N=2$, $0<\mu<2$, $(V-1)$ with $V=V_0$ and $(f_1)-(f_5)$ with
$\tilde{ f}=f$. Then, $(P)$ has a ground state solution in $H^{1}(\R^2)$.
\end{thm}

Concerning with the problem where the nonlinearity is asymptotically periodic, our main result is the following
\begin{thm}\label{MRS2}
Assume $N=2, 0<\mu<2$, $(V)$  and $(f_1)-(f_6)$. Then, $(P)$ has a ground state solution in $H^{1}(\R^2)$.
\end{thm}

In this paper, we will use  $C$, $C_i$ to denote positive constants and $B_R$ will denote the open ball centered at the origin with
radius $R>0$.  If $E$ is a real Hilbert space and $I:E \to \R$ is a functional of class ${C}^1(E,\mathbb{R})$, we say
that $(u_n)\subset E$ is a  $(PS)_c$ sequence for $I$, when $(u_n)$ satisfies
$$
I(u_n)\to c \,\,\, \mbox{and} \,\,\,\, I'(u_n)\to 0  \,\,\, \mbox{as} \,\,\, n\to\infty.
$$
Moreover, we say that $I$ satisfies the $(PS)_c$, if any $(PS)_c$ sequence possesses a convergent subsequence.

To conclude this introduction, we would like to cite some recent works involving exponential critical growth for the elliptic problem of the form
$$
-\Delta{u}+V(x)u=f(x,u) \,\,\, \,\,\, \mbox{in} \,\,\, \mathbb{R}^{2}.
$$
See for example, Adimurthi and K. Sandeep \cite{AS}, Adimurthi and Yang \cite{AY}, Albuquerque, Alves and Medeiros \cite{AAM}, do \'O, Medeiros and Severo \cite{OMS}, do \'O and de Souza \cite{OS}, Li and Ruf \cite{YR} and their references.

\section{Mountain Pass Geometry}
Since we are going to study the existence of positive solution via variational method, we will assume that
$$
f(x,s)=0 \quad \forall s \leq 0 \,\,\, \mbox{and} \,\,\, \forall x \in \mathbb{R}^{2}.
$$
From $(f_1)-(f_3)$, for any $\vr>0$, $p\geq1$ and $\beta>1$, there exists $C_\vr>0$ such that
$$
|f(x,s)|\leq \vr |s|^{\frac{2-\mu}{2}}+C(\vr,p, \beta) |s|^{p-1}\big[ e^{ \beta4\pi s ^{2}}
-1 \big] \ \ \ \forall s\in \R,
$$
 and
$$
|F(x,s)|\leq \vr |s|^{\frac{4-\mu}{2}}+C(\vr,p, \beta) |s|^{p}\big[ e^{\beta4\pi s ^{2}}
-1 \big] \ \ \ \forall s\in \R.
$$
From Lemma \ref{Trudinger-Moser} and H\"{o}lder inequality, we deduce that $F(x,u)\in L^{\frac{4}{4-\mu}}(\R^2)$ for any $u\in H^{1}(\R^2)$. Then, applying Hardy-Littlewood-Sobolev inequality, with  $t=r=\frac{4}{4-\mu}$, we see that
$$
\Big( I_\mu\ast F(x,u)\Big)F(x,u)\in L^{1}(\R^2),
$$
and so, the energy functional $I:E \to \mathbb{R}$  associated with problem $(SNE)$ given by
$$
\aligned
I(u)=\frac12\int_{\mathbb{R}^2} (|\nabla u|^2+ V(x) |u|^2)-\frac{1}{2}\int_{\R^2}\Big( I_\mu\ast F(x,u)\Big)F(x,u)
\endaligned
$$
is well defined on $E$. Furthermore, $I \in C^{1}(E,\mathbb{R})$ and
$$
I'(u)\varphi=\int_{\mathbb{R}^2} (\nabla u\nabla \varphi+ V(x) u\varphi)-\int_{\R^2}\Big( I_\mu\ast F(x,u)\Big)f(x,u) \,\,\, \forall u, \varphi \in E.
$$

Next, we will show that $I$ verifies the Mountain Pass Geometry.
\begin{lem}\label{mountain:1}
Assume $0<\mu<2$, $(f_1)-(f_3)$ and $(V-1)$. Then,
\begin{itemize}
  \item[$(1).$] \quad There exist $\rho, \delta_0>0$ such that $I|_{S_\rho}\geq\delta_0>0$, $\forall u\in S_\rho=\{u\in E:\|u\|=\rho\}$.
  \item[$(2).$] \quad There is $e\in E$ with $\|e\|>\rho$ such that $I(e)< 0$.
\end{itemize}
\end{lem}
\begin{proof}
(1). For any $\vr>0$, $p>1$ and $\beta>1$, there exists $C_\vr>0$ such that
$$
|F(x, s)|\leq \vr |s|^{\frac{4-\mu}{2}}+C(\vr,p, \beta) |s|^{p}\big[ e^{ \beta4\pi s ^{2}}
-1 \big] \,\,\, \forall s\in \R,
$$
from where it follows
\begin{equation} \label{mp1}
|F(x,u)|_{\frac{4}{4-\mu}}\leq \vr C |u|^{\frac{4-\mu}{2}}_2+C(\vr,p, \beta)\big| u^{p}\big[ e^{ \beta4\pi u ^{2}}
-1 \big]\big|_{\frac{4}{4-\mu}}.
\end{equation}
Since the imbedding $E\hookrightarrow L^p(\R^2)$ is continuous, for each $p\in (2,+\infty)$, there exists a constant $C_1>0$ such that
$$\aligned
\int_{\R^2}|u|^{\frac{4p}{4-\mu}}\big[ e^{\beta4\pi u ^{2}}
-1 \big]^\frac{4}{4-\mu}&\leq (\int_{\R^2}|u|^{\frac{8p}{4-\mu}})^{\frac12}(\int_{\R^2}\big[ e^{ \beta4\pi u ^{2}}
-1 \big]^\frac{4}{4-\mu})^{\frac12}\\
&\leq C_1\|u\|^{\frac{4p}{4-\mu}}\big(\int_{\R^2}\big[ e^{(\frac{4\beta}{4-\mu}4\pi u ^{2})}
-1 \big]\big)^{\frac12}.
\endaligned
$$
Observing that
$$
\int_{\R^2}\big[ e^{(\frac{4\beta}{4-\mu}4\pi u ^{2})}
-1 \big]=\int_{\R^2}\big[ e^{(\frac{4\beta}{4-\mu}\|u\|^2 4\pi \frac{u ^{2}}{\|u\|^2})}
-1 \big],
$$
fixing $\xi \in (0,1)$ and $\frac{4\beta}{4-\mu}\|u\|^2=\xi< 1$, the Lemma \ref{Trudinger-Moser} gives
$$
\int_{\R^2}\big[ e^{\xi 4\pi \frac{u ^{2}}{\|u\|^2}}
-1 \big] \leq C_2 \,\,\,\,\, \mbox{for} \,\,\,\,\, \|u\|=\left(\frac{\xi(4-\mu)}{4\beta}\right)^{\frac{1}{2}},
$$
for some positive constant $C_2$. Gathering the last estimate and \eqref{mp1}, there exists $C_3>0$ such that
$$
|F(x,u)|_{\frac{4}{4-\mu}}\leq \vr \|u\|^{\frac{4-\mu}{2}}+C_3\|u\|^{p} \,\,\,\,\, \mbox{for} \,\,\,\,\, \|u\|=\left(\frac{\xi(4-\mu)}{4\beta}\right)^{\frac{1}{2}}.
$$
Thereby, by Hardy-Littlewood-Sobolev inequality,
$$
\int_{\R^2}\Big( I_\mu\ast F(x,u)\Big)F(x,u)\leq \vr^2 C \|u\|^{4-\mu}+2C_3\|u\|^{2p} \,\,\,\,\, \mbox{for} \,\,\,\,\, \|u\|=\left(\frac{\xi(4-\mu)}{4\beta}\right)^{\frac{1}{2}},
$$
and so,
$$
\aligned
I(u) &\geq  \frac12\|u\|^2-\vr^2 C \|u\|^{4-\mu}-C_3\|u\|^{2p}.
\endaligned
$$
Since $0<\mu<2$ and $p>1$,  $(1)$ follows choosing  $\rho=\Big(\frac{\xi (4-\mu)}{4\beta}\Big)^{\frac{1}{2}}$ with $\xi \approx 0^+$.

(2) \quad  Fixing $u_0 \in E$ with $u_0^{+}(x)=\max\{u_0(x),0\} \not= 0$, we set
$$
\mathcal{A}(t)=\Psi(\frac{tu_0}{\|u_0\|})>0 \,\,\ \mbox{for} \,\,\, t>0,
$$
where
$$
\Psi(u)=\frac12\int_{\R^2}\Big( I_\mu\ast F(x,u)\Big)F(x,u).
$$
A straightforward  computation yields
$$
\frac{\mathcal{A}'(t)}{\mathcal{A}(t)}\geq \frac{\theta}{t} \,\,\, \mbox{for all} \,\,\, t>0.
$$
Then, integrating this over $[1, s\|u_0\|]$ with $s>\frac{1}{\|u_0\|}$,  we find
$$
\Psi(su_0)\geq \Psi(\frac{u_0}{\|u_0\|})\|u_0\|^{\theta} s^{\theta}.
$$
Therefore
$$
\Phi(su_0)\leq C_1 s^2-C_2s^{\theta} \,\,\, \mbox{for} \,\,\ s > \frac{1}{\|u_0\|},
$$
and $(2)$ holds for $e=s u_0$ with $s$ large enough.
\end{proof}
By the Mountain Pass Theorem without $(PS)$ condition found in \cite{MW}, there is a $(PS)_{c_V}$ sequence $(u_n) \subset E$, that is,
\begin{equation} \label{pss}
I(u_n) \to c_V \,\,\,\, \mbox{and} \,\,\,\, I'(u_n) \to 0,
\end{equation}
where $c_V$ is the mountain pass level characterized by
\begin{equation} \label{m}
0<c_V:=\inf_{\gamma\in \Gamma} \max_{t\in [0,1]}
I(\gamma(t))
\end{equation}
with
$$
\Gamma:=\{\gamma\in \mathcal{C}^1([0,1], E):\gamma(0)=0\ \ \hbox{and}   \ \  I(\gamma(1))<0\}.$$

The next lemma is crucial in our arguments, because it establishes an important estimate involving the level $c_V$.

\begin{lem}\label{BD}
The mountain pass level $c_V$ satisfies $c_V\in [\rho,\frac{(2-\mu)(\theta-2)}{8\theta})$. Moreover, the $(PS)_{c_V}$ sequence $(u_n)$ is bounded and its weak limit $u$ satisfies $I'(u)=0$.
\end{lem}
\begin{proof}
From $(f_3)$,
$$
c_V=
\lim\Big(I(u_n)-\frac{1}{\theta}I'(u_n)u_n\Big)\geq (\frac12-\frac{1}{\theta})\limsup\|u_n\|^2
$$
which means
\begin{equation} \label{mp2}
\limsup\|u_n\|^2\leq \frac{2\theta}{\theta-2}c_V.
\end{equation}

Let $u_p\in E$ be a positive radial function verifying
$$
\displaystyle S_p=\displaystyle\inf_{u\in E, u\neq0}\frac{\displaystyle\left(\int_{\R^2}(|\nabla u_p|^2+ |V|_{\infty}|u_p|^2)\right)^{1/2}}{\displaystyle\left(\int_{\R^2}\Big( I_\mu\ast |u_p|^p\Big)|u_p|^p\right)^{\frac{1}{2p}}}.
$$
By $(f_4)$, it is easy to see that
$$
\aligned
c_V&=\inf_{\gamma\in \Gamma} \max_{t\in [0,1]}
I(\gamma(t))\\
&\leq\inf_{u\in E\backslash\{0\}} \max_{t\geq 0}I(tu)\\
&\leq \max_{t\geq 0}I(tu_p)\\
&\leq \max_{t\geq 0}\Big\{\frac{t^2}{2}\int_{\mathbb{R}^2} (|\nabla u_p|^2+ |V|_{\infty} |u_p|^2)-\frac{t^{2p}C^2_p}{2}\int_{\R^2}\Big( I_\mu\ast |u_p|^p\Big)|u_p|^p\Big\}\\
&=\frac{(p-1)S^{\frac{2p}{p-1}}_p}{2p^{\frac{p}{p-1}}C^{\frac{2}{p-1}}_p}\\
&<\frac{(2-\mu)(\theta-2)}{8\theta}.
\endaligned
$$
Consequently, from \eqref{mp2},
$$
\limsup\|u_n\|^2<\frac{(2-\mu)}{4}.
$$
\end{proof}

In the following, we may assume that there are $n_0 \in \mathbb{N}$ and $m \in (0,  \frac{(2-\mu)}{4})$, such that
\begin{equation} \label{ZRX1}
\|u_n\|^2 \leq m \,\,\, \forall n \geq n_0.
\end{equation}
Without lost of generality, in what follows we suppose that $n_0=1$.

\begin{Claim}\label{BNT}
There exists $C>0$ such that
$$
|I_\mu\ast F(x,u_n)|_{\infty}<C \ \ \forall n\in \N.
$$
\end{Claim}
\begin{proof}
For each $\beta>1$, there exists $C_0>0$ such that

$$
F(x,s)\leq C_0\Big( |s|^{\frac{4-\mu}{2}}+|s|\big[ e^{\beta4\pi s ^{2}}
-1 \big]\Big) \,\,\, \forall s\in \R.
$$
Hence,
$$\aligned
|I_\mu\ast F(x,u_n)(x)|&\leq\Big|\int_{\R^2}\frac{F(x,u_n)}{|x-y|^\mu}\Big|\\
&=\Big|\int_{|x-y|\leq1}\frac{F(x,u_n)}{|x-y|^\mu}\Big|+\Big|\int_{|x-y|\geq1}\frac{F(x,u_n)}{|x-y|^\mu}\Big|\\
&\leq C_0\int_{|x-y|\leq1}\frac{|u_n|^{\frac{4-\mu}{2}}+|u_n|\big[ e^{\beta4\pi |u_n|^{2}}
-1 \big]}{|x-y|^\mu}\\
&\hspace{5mm}+C_0\int_{|x-y|\geq1}\Big(|u_n|^{\frac{4-\mu}{2}}+|u_n|\big[ e^{\beta4\pi |u_n|^{2}}
-1 \big]\Big).
\endaligned
$$
Since
$$
\frac{1}{|y|^{\mu}} \in L^{\frac{2+\delta}{\mu}}(B_1^{c}(0)) \,\,\,\, \forall ~~ \delta >0,
$$
we take $\delta \approx 0^{+}$ verifying
$$
q_{1,\delta}=\frac{(4-\mu)}{2}\frac{(2+\delta)}{(2+\delta)-\mu}>2.
$$
Using H\"{o}lder inequality, we derive that
\begin{equation} \label{ES1}
\int_{|x-y|\geq1}\frac{|u_n|^{\frac{4-\mu}{2}}}{|x-y|^\mu} \leq C_0\left(\int_{|x-y|\geq 1}|u_n|^{q_{1,\delta}} \right)^{\frac{(2+\delta)-\mu}{2+\delta}}\leq C_1 \,\,\, \forall n \in \mathbb{N}.
\end{equation}
By (\ref{ZRX1}), we can fix $\beta>1$ close to $1$, such that $2 \beta m \in (0,1)$. Then, by Trudinger-Moser inequality,  there exists $C_2>0$ such that
\begin{equation} \label{ES2}
\int_{|x-y|\geq1}|u_n|\big[ e^{\beta4\pi u_n^{2}}
-1 \big]\leq |u_n|_2\int_{\R^2}\Big(\big[ e^{2\beta m4\pi\frac{ u_n^{2}}{\|u_n\|^2}}
-1 \big]\Big)^{\frac12}\leq C_2 \,\,\,\, \forall n \in \mathbb{N}.
\end{equation}
Choosing  $t\in \big(\frac{2}{2-\mu}, +\infty)$, we see that $\frac{(4-\mu)t}{2}>2$ and $1-\frac{t\mu}{t-1}>-1$. Thus, by H\"{o}lder inequality,
\begin{equation} \label{ES3}
\aligned
\int_{|x-y|\leq1}\frac{|u_n|^{\frac{4-\mu}{2}}}{|x-y|^\mu}&\leq \big(\int_{|x-y|\leq1}|u_n|^{\frac{(4-\mu)t}{2}}\big)^{\frac1t}\big(\int_{|x-y|\leq1}\frac{1}{|x-y|^{\frac{t\mu}{t-1}}}\big)^{\frac{t-1}{t}}\\
&\leq  C_2\big(\int_{|r|\leq1}{|r|^{1-\frac{t\mu}{t-1}}}dr\big)^{\frac{t-1}{t}} \leq C_3 \,\,\,\, \forall n \in \mathbb{N},
\endaligned
\end{equation}
for some $C_3>0$. Now, for $t>\frac{2}{2-\mu}$ and close to $\frac{2}{2-\mu}$,  we can assume that $2 \beta t m \in (0,1)$. Thus, the Trudinger-Moser inequality and the boundedness of $(u_n)$ in $E$ combine to give
$$\aligned
\int_{|x-y|\leq1}&\frac{|u_n|\big[ e^{\beta4\pi u_n^{2}}
-1 \big]}{|x-y|^\mu}\\
&\leq \big(\int_{|x-y|\leq1}|u_n\big[ e^{\beta4\pi u_n^{2}}
-1 \big]|^{t}\big)^{\frac1t}\big(\int_{|x-y|\leq1}\frac{1}{|x-y|^{\frac{t\mu}{t-1}}}\big)^{\frac{t-1}{t}}\\
&\leq  \big(\int_{|x-y|\leq1}|u_n|^{2t}\big)^{\frac{1}{2t}}\big(\int_{|x-y|\leq1}\big[ e^{ 2\beta tm4\pi \frac{u_n^{2}}{\|u_n\|^2}}
-1 \big]\big)^{\frac{1}{2t}}\big(\int_{|r|\leq1}{|r|^{1-\frac{t\mu}{t-1}}}dr\big)^{\frac{t-1}{t}}\\
\endaligned
$$
implying that there is $C_4>0$ such that
\begin{equation} \label{ES4}
\int_{|x-y|\leq1}\frac{|u_n|\big[ e^{\beta4\pi u_n^{2}}
-1 \big]}{|x-y|^\mu} \leq C_4 \,\,\, \forall n \in \mathbb{N}.
\end{equation}
Now, the claim follows from (\ref{ES1})-(\ref{ES4}). \end{proof}

\begin{Claim}\label{WL}
Let $(u_n)$ be the $(PS)_{c_V}$ sequence with weak limit $u$. Then, $u$ satisfies $I'(u)=0$.
\end{Claim}
\begin{proof}
Since $(u_n)$ is bounded in $E$, going to a subsequence still denoted by $(u_n)$, there is  $u\in E$ such that
$$
u_n\rightharpoonup u \,\,\, \mbox{in} \,\,\, E, u_n \to u \,\,\, \mbox{in} \,\,\, L^q_{loc}(\R^2) \,\,\, \forall q \in [1,+\infty) \,\,\, \mbox{and} \,\,\, u_n(x)\to u(x) \,\,\, \mbox{a.e. in} \,\, \R^2.
$$

For each $\varphi\in C_{0}^{\infty}(\R^2)$, the Hardy-Littlewood-Sobolev inequality together with H\"{o}lder inequality lead to
\begin{equation} \label{mp3}
\aligned
\Big|\int_{\R^2}\Big( I_\mu\ast F(x,u_n)\Big)f(x,u_n)\varphi \Big|&\leq |F(x,u_n)|_{\frac{4}{4-\mu}}|f(x,u_n)|_{\frac{4r}{4-\mu}}|\varphi|_{\frac{4r'}{4-\mu}}
\endaligned
\end{equation}
where $r=\frac{4-\mu}{2-\mu}$ and $r'>1$ satisfies $\frac{1}{r}+\frac{1}{r'}=1$. From $(f_1)-(f_3)$, for each $\vr>0$ and $\beta>1$, there is $C(\vr, \beta)>0$ such that
$$
|f(x,s)|\leq \vr |s|^{\frac{2-\mu}{2}}+C(\vr, \beta)\big[ e^{ \beta4\pi s ^{2}}
-1 \big] \,\,\,\,  \forall s\in \R,
$$
 and
$$
|F(x,s)|\leq \vr |s|^{\frac{4-\mu}{2}}+C(\vr,\beta) |s|\big[ e^{ \beta4\pi s ^{2}}
-1 \big] \,\,\,\, \forall s\in \R.
$$
Then
$$
\aligned
|f(x,u_n)|_{\frac{4r}{4-\mu}}&\leq \vr |u_n|^{\frac{2-\mu}{2}}_{2}+C(\vr, \beta)|e^{ \beta4\pi u_n ^{2}}
-1|_{\frac{4r}{4-\mu}}\\
&\leq  C_1 \|u_n\|^{\frac{2-\mu}{2}}+C_1\big(\int_{\R^2}[e^{( \frac{4\beta r}{4-\mu}\|u_n\|^2 4\pi \frac{u_n ^{2}}{\|u_n\|^2})}
-1]\big)^{\frac{4-\mu}{4r}}.
\endaligned
$$
Now, choosing $\beta>1$ sufficiently close to $1$,  the estimate (\ref{ZRX1}) and Trudinger-Moser combined give that $(f(x,u_n))$ is bounded in $L^{\frac{4}{2-\mu}}(\mathbb{R}^{2})$. Moreover, with a similar argument, the sequence $(F(x,u_n))$ is also bounded $L^{\frac{4}{4-\mu}}(\mathbb{R}^{2})$.

In the sequel, we will prove that for any $\varphi\in C^{\infty}_{0}(\mathbb{R}^{2})$, there limit below holds
$$
\int_{\R^2}\Big( I_\mu\ast F(x,u_n)\Big)f(x,u_n)\varphi \to \int_{\R^2}\Big( I_\mu\ast F(x,u)\Big)f(x,u)\varphi.
$$
In fact, for any $\varphi\in C^{\infty}_{0}(\mathbb{R}^{2})$,
\begin{equation} \label{mp6}
\aligned
\Big|\int_{\R^2}\Big(\big( I_\mu\ast F(x,u_n)\big)&f(x,u_n)-\big( I_\mu\ast F(x,u)\big)f(x,u)\Big)\varphi\Big|\\
&\leq\Big|\int_{\R^2}\big( I_\mu\ast F(x,u_n)\big)\Big(f(x,u_n)-f(x,u)\Big)\varphi\Big|\\
&\hspace{5mm}+\Big|\int_{\R^2}\Big(I_\mu\ast \big(F(x,u_n)-F(x,u)\big)\Big)f(x,u)\varphi\Big|.
\endaligned
\end{equation}

For the above first term, we recall that $(I_\mu\ast F(x,u_n))$ is bounded in $L^{\infty}(\R^2)$. Then, 
$$
\aligned
\Big|\int_{\R^2}\big( I_\mu\ast F(x,u_n)\big)&\Big(f(x,u_n)-f(x,u)\Big)\varphi\Big|\\
&\leq C\Big|\int_{\R^2}\big(f(x,u_n)-f(x,u)\big)\varphi\Big|.
\endaligned
$$
Since $u_n(x)\to u(x)$ a.e. in $\R^2$, the continuity of $f$ implies $f(x,u_n(x))\to f(x,u(x))$  a.e. in $\R^2$. This fact combined with boundedness of 
$(f(x,u_n))$ in $L^{\frac{4}{2-\mu}}(\R^2)$ leads to
$$
f(x,u_n)\rightharpoonup f(x,u) \  \  \hbox{in}\ \ L^{\frac{4}{2-\mu}}(\R^2),
$$
from where it follows that
$$
\int_{\R^2}\big(f(x,u_n)-f(x,u)\big)\varphi \to 0.
$$
Consequently,
\begin{equation} \label{mp7}
\Big|\int_{\R^2}\big( I_\mu\ast F(x,u_n)\big)\Big(f(x,u_n)-f(x,u)\Big)\varphi\Big|\to 0,
\end{equation}
for any $\varphi\in C^{\infty}_{0}(\mathbb{R}^{2})$.

For the second term, notice that
$$
\aligned
\Big|\int_{\R^2}\Big(I_\mu\ast\big(F(x,u_n)&-F(x,u)\big)\Big)f(x,u)\varphi\Big|\\
&=\Big|\int_{\R^2}\big(F(x,u_n)-F(x,u)\big)I_\mu\ast(f(x,u)\varphi)\Big|
\endaligned
$$
Since $u_n(x)\to u(x)$ a.e. in $\R^2$, the continuity of $F$ implies $F(x,u_n(x))\to F(x,u(x))$  a.e. in $\R^2$. Using the boundedness of
$(F(x,u_n))$ in $L^{\frac{4}{4-\mu}}(\R^2)$, we conclude that
$$
F(x,u_n)\rightharpoonup F(x,u) \  \  \hbox{in}\ \ L^{\frac{4}{4-\mu}}(\R^2).
$$
As
$$
I_\mu\ast(f(x,u)\varphi) \in L^{\frac{4}{\mu}}(\R^2),
$$
we must have,
\begin{equation} \label{mp8}
\Big|\int_{\R^2}\Big(I_\mu\ast\big(F(x,u_n)-F(x,u)\big)\Big)f(x,u)\varphi\Big|\to 0,
\end{equation}
for any $\varphi\in C_{0}^{\infty}(\R^2)$. Now, the result follows by using the density of $C_{0}^{\infty}(\R^2)$ in $H^{1}(\R^2)$. \end{proof}

\section{Proof of the main results}

In this section, we will prove the Theorems \ref{MRS1} and \ref{MRS2}.

\subsection{Proof of Theorem \ref{MRS1}.}
Let $(u_n)$ be the $(PS)_{c_V}$ sequence. Since $(u_n)$ is bounded and $
\limsup\|u_n\|^2<\frac{(2-\mu)}{4}
$, we have either $(u_{n})$ is vanishing,
 i.e., there exists $r>0$ such that
$$
\limsup_
 {y\in \R^2}\int_{B_r(y)}|u_{n}|^2=0
 $$
 or non-vanishing, i.e., there exist $r, \delta>0$ and a sequence $(y_n) \subset \mathbb{Z}^2$ such that
 $$
\lim_{n\to\infty}\int_{B_r(y_n)}|u_{n}|^2\geq\delta.
 $$

 If $(u_{n})$ is vanishing, then by Lion's result, we have that
 \begin{center}
 $u_{n}\to 0$ in
 $L^s(\R^2)$, \ \ $2<s< +\infty$.
 \end{center}
 Using Hardy-Littlewood-Sobolev inequality and $(f_3)$, we derive
$$
\aligned
\Big|\int_{\R^2}\Big( I_\mu\ast F(x,u_n)\Big)f(x,u_n)u_n \Big|&\leq C|F(x,u_n)|_{\frac{4}{4-\mu}}|f(x,u_n)u_n|_{\frac{4}{4-\mu}}\leq C|f(x,u_n)u_n|^2_{\frac{4}{4-\mu}}.
\endaligned
$$
For any $\vr>0$, $p> 1$ and $\beta>1$, there exists $C(\vr,p, \beta)>0$ such that
$$
|f(x,s)|\leq \vr |s|^{\frac{2-\mu}{2}}+C(\vr,p, \beta) |s|^{p-1}\big[ e^{ \beta4\pi s ^{2}}
-1 \big] \,\,\, \forall s\in \R.
$$
Then,
$$
|f(x,u_n)u_n|_{\frac{4}{4-\mu}}\leq \vr |u_n|_2^{\frac{4-\mu}{2}}+C(\vr,p, \beta)|u_n|_{\frac{4p t'}{4-\mu}}^{\frac{4-\mu}{4t'}}\big(\int_{\R^2}[e^{( \frac{4\beta t}{4-\mu}\|u_n\|^2 4\pi \frac{u_n ^{2}}{\|u_n\|^2})}
-1]\big)^{\frac{4-\mu}{4t}}
$$
where $t, t'>1$ satisfying $\frac{1}{t}+\frac{1}{t'}=1$. Now, gathering (\ref{ZRX1}) and Trudinger-Moser inequality, if $\beta, t>1$ are fixed close to $1$, we deduce that
$$
\big(\int_{\R^2}[e^{( \frac{4\beta t}{4-\mu}\|u_n\|^2 4\pi \frac{u_n ^{2}}{\|u_n\|^2})}
-1]\big)^{\frac{4-\mu}{4t}}\leq \big(\int_{\R^2}[e^{( \frac{4\beta m t}{4-\mu} 4\pi \frac{u_n ^{2}}{\|u_n\|^2})}
-1]\big)^{\frac{4-\mu}{4t}} \leq C_1 \,\,\,\, \forall n \in \mathbb{N},
$$
for some $C_1>0$. Then,
$$
\aligned
\Big|\int_{\R^2}\Big( I_\mu\ast F(x,u_n)\Big)f(x,u_n)u_n \Big|&\leq \vr |u_n|_2^{\frac{4-\mu}{2}}+C_2|u_n|_{\frac{4p t'}{4-\mu}}^{\frac{4-\mu}{4t'}}.
\endaligned
$$
Since $t>1$ is close to $1$, we have that $\frac{4p t'}{4-\mu}>2$. Consequently
$$
\int_{\R^2}\Big( I_\mu\ast F(x,u_n)\Big)f(x,u_n)u_n\to 0,
$$
implying that
$$
u_n\to0 \ \ \hbox{in} \ \ E, \ \ n\to \infty.
$$
Recalling that $I$ is a continuous functional, we must have
$$
I(u_n) \to 0,
$$
from where it follows that $c_V=0$, which  is a contradiction. Thereby, vanishing case does not hold.

From now on, we set $v_{n}=u_{n}(\cdot-y_n)$. Therefore, $\|v_n\|=\|u_n\|$ and
 $$
\int_{B_r(0)}|v_{n}|^2\geq\delta.
 $$
Using the definition of $I$, we see that $I$ and $I'$ are both invariant
 by $\mathbb{Z}^{2}$-translation. Then,
 $$
I(v_n) \to c_V \,\,\, \mbox{and} \,\,\, I'(v_{n})\to 0.
 $$
Since $(v_{n})$ is also bounded, we may assume
$v_{n}\rightharpoonup v$ in $E$ and $v_{n}\to v$ in
$L^2_{loc}(\R^2)$ . From the last inequality $v\neq0$, and  by the same arguments in Lemma \ref{BD} we can assume that $I'(v)=0$

Let $\cal{N}$ be the Nehari manifold defined by
$$
{\cal{N}}=\{u\in E:u\neq0, I'(u)u=0\}.
$$
From $(f_5)$, it is standard to check that the mountain pass level can be characterized by
$$
c_V=\inf_{u\in E\{0\}} \max_{t\geq0}
I(t u)=\inf_{u\in \cal{N}}I(u).
$$

The above characterization together with $(f_3)$ give $I(v)=c_V$, showing that $v$ is a ground state solution.

\subsection{Proof of Theorem \ref{MRS2}.}

In what follows, we will denote by $\tilde{I}:H^{1}(\mathbb{R}^{2}) \to \mathbb{R}$ the energy functional associated with problem
$$
\left\{\aligned &-\Delta u +V_0(x)u =\Big( I_\mu\ast \tilde{F}(x,u)\Big)\tilde{f}(x,u)  \quad \mbox{in} \quad \R^2, \\
&  u\in H^{1}({\R}^2),\\
& u(x)>0\ \ \mbox{for all}\ \ x\in\mathbb{R}^2,
\endaligned\right.
$$
where $0<\mu<2$, $\tilde{F}(x,s)$ is the primitive function of $\tilde{f}(x,s)$ in the variable $s$ and $V_0,\tilde{f}$ are continuous and 1-periodic. Thus,
$$
\aligned
\tilde{I}(u)=\frac12\int_{\mathbb{R}^2} (|\nabla u|^2+ V_0(x) |u|^2)-\frac{1}{2}\int_{\R^2}\Big( I_\mu\ast \tilde{F}(x,u)\Big)\tilde{F}(x,u) \,\,\, \forall u \in H^{1}(\mathbb{R}^{2}).
\endaligned
$$
By Theorem \ref{MRS2}, we know that there is $u_0 \in H^{1}(\mathbb{R}^{2})$ such that
$$
\tilde{I}'(u_0)=0 \,\,\, \mbox{and} \,\,\, \tilde{I}(u_0)=\tilde{c}_V,
$$
where $\tilde{c}_V$ is the mountain pass level associated with $\tilde{I}$.

Using the same arguments explored in proof of Lemma \ref{mountain:1}, we can show that ${I}$ verifies the Mountain Pass Geometry. Consequently,
 there is a (PS) sequence $(u_n) \subset E$ such that
\begin{equation} \label{pss}
{I}(u_n)\to {c}_V    \,\,\, \mbox{and} \,\,\,    {I}'(u_n)\to0 ,
\end{equation}
where ${c}_V$ is the mountain pass level characterized by
\begin{equation} \label{m}
0<{c}_V:=\inf_{\gamma\in \Gamma} \max_{t\in [0,1]}
{I}(\gamma(t))=\inf_{u\in E\{0\}} \max_{t\geq0}
{I}(t u)=\inf_{u\in \cal{N}}{I}(u)
\end{equation}
with
$$
\Gamma:=\{\gamma\in \mathcal{C}([0,1], E):\gamma(0)=0\ \ \hbox{and}   \ \  {I}(\gamma(1))<0\}$$
and
$$
{\cal{{N}}}=\{u\in E:u\neq0, {I}'(u)u=0\}.
$$

Therefore, by the above notations,
$$
{c}_V=\inf_{u\in E\{0\}} \max_{t\geq0}
{I}(t u)\leq \max_{t\geq0}
{I}(t u_0)={I}(t_0u_0)<\tilde{I}(t_0u_0)=\max_{t\geq0}\tilde{I}(t u_0)=\tilde{I}(u_0)=\tilde{c}_V,
$$
that is,
\begin{equation} \label{NOVA}
{c}_V < \tilde{c}_V.
\end{equation}
From Lemma \ref{BD},
$$
\tilde{c}_V\in (0,\frac{(2-\mu)(\theta-2)}{8\theta})
$$
then,
$$
{c}_V\in (0,\frac{(2-\mu)(\theta-2)}{8\theta}).
$$
Recalling that $\theta \geq \tilde{\theta}$, it follows from  $(f_3)$,
$$
\tilde{\theta} F(x,s) \leq f(x,s)s \,\,\ \forall x \in \mathbb{R}^{2} \,\,\, \mbox{and} \,\,\, s \in \mathbb{R}.
$$
Arguing as in the proof Lemma \ref{BD}, it follows that $(u_n)$ is bounded in $E$ with
$$
\limsup\|u_n\|^2<\frac{(2-\mu)}{4}.
$$
Moreover, there is $u \in E$ such that for a subsequence,
$$
u_n\rightharpoonup u \,\,\, \mbox{in} \,\,\, E \,\,\, \mbox{and} \,\,\, \tilde{I}'(u)=0.
$$
We claim that $u\neq0$.  To see why, we will argue by contradiction, supposing  that $u=0$. Notice that, for any $\varphi\in E$,
$$
{I}(u_n)=\tilde{I}(u_n)+\frac{1}{2}\int_{\R^2}\Big( I_\mu\ast \tilde{F}(x,u_n)\Big)\tilde{F}(x,u_n)-\frac{1}{2}\int_{\R^2}\Big( I_\mu\ast{F}(x,u_n)\Big){F}(x,u_n)
$$
and
$$
{I}'(u_n)\varphi=\tilde{I}'(u_n)\varphi+\int_{\R^2}\Big( I_\mu\ast \tilde{F}(x,u_n)\Big)\tilde{f}(x,u_n)\varphi-\int_{\R^2}\Big( I_\mu\ast {F}(x,u_n)\Big){f}(x,u_n)\varphi.
$$
By $(f_6)$, it is possible to prove that
\begin{equation} \label{WWW1}
\int_{\R^2}\Big( I_\mu\ast \tilde{F}(x,u_n)\Big)\tilde{F}(x,u_n)-\int_{\R^2}\Big( I_\mu\ast{F}(x,u_n)\Big){F}(x,u_n)\to 0
\end{equation}
and
\begin{equation} \label{WWW2}
\int_{\R^2}\Big( I_\mu\ast \tilde{F}(x,u_n)\Big)\tilde{f}(x,u_n)\varphi-\int_{\R^2}\Big( I_\mu\ast {F}(x,u_n)\Big){f}(x,u_n)\varphi\to 0
\end{equation}
uniformly for $\|\varphi\|\leq1, \varphi\in E$. Next, we will show only (\ref{WWW2}), because (\ref{WWW1}) follows with the same type of argument. For any $\varphi\in E$, considering $r=\frac{4-\mu}{2-\mu}$ and $r'>1$ satisfying $\frac{1}{r}+\frac{1}{r'}=1$, we have
$$\aligned\Big|\int_{\R^2}\Big( I_\mu\ast \tilde{F}(x,u_n)\Big)&\tilde{f}(x,u_n)\varphi-\int_{\R^2}\Big( I_\mu\ast {F}(x,u_n)\Big){f}(x,u_n)\varphi\Big|\\
&\leq\Big|\int_{\R^2}\Big( I_\mu\ast (\tilde{F}(x,u_n)-F(x,u_n))\Big)\tilde{f}(x,u_n)\varphi\Big|\\
&\hspace{5mm}+\Big|\int_{\R^2}\Big( I_\mu\ast {F}(x,u_n)\Big)(\tilde{f}(x,u_n)-{f}(x,u_n))\varphi\Big|\\
&\leq|\tilde{F}(x,u_n)-F(x,u_n)|_{\frac{4}{4-\mu}}|f(x,u_n)|_{\frac{4r}{4-\mu}}|\varphi|_{2}\\
&\hspace{5mm}+|{F}(x,u_n)|_{\frac{4}{4-\mu}}|\tilde{f}(x,u_n)-{f}(x,u_n)|_{\frac{4r}{4-\mu}}|\varphi|_{2}\\
&\leq C |\tilde{F}(x,u_n)-F(x,u_n)|_{\frac{4}{4-\mu}}|f(x,u_n)|_{\frac{4r}{4-\mu}}\\
&\hspace{5mm}+C |{F}(x,u_n)|_{\frac{4}{4-\mu}}|\tilde{f}(x,u_n)-{f}(x,u_n)|_{\frac{4r}{4-\mu}}.
\endaligned
$$
Since $\limsup \|u_n\|^2< \frac{2-\mu}{4}$, the ideas used in previous section work to show that $(\tilde{F}(x,u_n))$ and $(\tilde{f}(x,u_n))$ are bounded in $L^{\frac{4}{4-\mu}}(\R^2)$ and $L^{\frac{4r}{4-\mu}}(\R^2)$ respectively.

On the other hand, from $(f_6)$, there exist $C>0$ such that
$$
|\tilde{f}(x,s)-{f}(x,s)|\leq C A(x)\Big(|s|^{\frac{2-\mu}{2}}+s\big[ e^{ \beta4\pi s ^{2}}
-1 \big]\Big) \,\,\,\,  \forall s\in \R  \,\,\, \mbox{and} \,\,\, x \in \R^2,
$$
leading to
$$
\aligned
|\tilde{f}(x,u_n)-{f}(x,u_n)|_{\frac{4}{2-\mu}}&\leq C(\int_{\R^2}|A(x)|^{\frac{4}{2-\mu}}|u_n|^{2})^{\frac{2-\mu}{4}}\\
&\hspace{5mm}+C(\int_{\R^2}|A(x)u_n|^{\frac{4}{2-\mu}}[e^{ \frac{4\beta}{2-\mu}4\pi u_n ^{2}}
-1])^{\frac{2-\mu}{4}}\\
&\leq C(\int_{\R^2}|A(x)|^{\frac{4}{2-\mu}}|u_n|^{2})^{\frac{2-\mu}{4}}\\
&\hspace{5mm}+C(\int_{\R^2}|A(x)u_n|^{\frac{4t'}{2-\mu}})^{\frac{2-\mu}{4t'}}(\int_{\R^2}[e^{( \frac{4\beta mt}{4-\mu} 4\pi \frac{u_n ^{2}}{\|u_n\|^2})}
-1])^{\frac{2-\mu}{4t}}
\endaligned
$$
Fixing $t>1$ sufficiently close to $1$, again by Lemma \ref{Trudinger-Moser},
$$
(\int_{\R^2}[e^{( \frac{4\beta mt}{4-\mu} 4\pi \frac{u_n ^{2}}{\|u_n\|^2})}
-1])^{\frac{2-\mu}{4t}}<C \,\,\, \forall n \in \mathbb{N},
$$
for some $C>0$. Since $A \in L^{\infty}(\mathbb{R}^{2})$ and $A(x)\to 0$ as $|x|\to \infty$, it easy to obtain
$$
(\int_{\R^2}|A(x)|^{\frac{4}{2-\mu}}|u_n|^{2})^{\frac{2-\mu}{4}}\to 0
$$
and
$$
(\int_{\R^2}|A(x)u_n|^{\frac{4t'}{2-\mu}})^{\frac{2-\mu}{4t'}}\to 0.
$$
Therefore
$$
|\tilde{f}(x,u_n)-{f}(x,u_n)|_{\frac{4}{2-\mu}}\to 0,
$$
and by a similar argument,
$$
|\tilde{F}(x,u_n)-F(x,u_n)|_{\frac{4}{4-\mu}}\to 0,
$$
showing (\ref{WWW2}).

Consequently, the sequence $(u_n) \subset E$ satisfies
\begin{equation} \label{ps2}
\tilde{I}'(u_n)\to0 \,\,\, \mbox{and} \,\,\, \tilde{I}(u_n)\to.
{{c}}_V,
\end{equation}
Repeating the arguments explored in the proof of Theorem \ref{MRS1}, we will find a nontrivial critical point $\tilde{u}$ of $\tilde{I}$ verifying the estimate $\tilde{I}(\tilde{u}) \leq {{c}}_V$. However, this is a contradiction, because by definition of  ${\tilde{c}}_V$, we must have
$$
{\tilde{c}}_V \leq \tilde{I}(\tilde{u})
$$
implying that $\tilde{c}_V \leq {c}_V$, which is a absurd with (\ref{NOVA}). Thereby, the weak limit $u$ of the $(PS)_{c_V}$ is nontrivial, finishing the proof of Theorem  \ref{MRS2}.

\section{Final comments}

The positive solution $u$ obtained in Theorem \ref{MRS1} belongs to $L^{\infty}(\R^2)$ and decays to zero as $|x|\to \infty$. First of all, we would like point out that the arguments used in the proof of Claim \ref{BNT} implies that 
$$
I_\mu\ast F(x,v) \in L^{\infty}(\mathbb{R}^{2}) \,\,\, \forall v \in H^{1}(\mathbb{R}^{2}). 
$$
Thus, if  $u$ is a solution of
$$
-\Delta u +V(x)u=\Big( I_\mu\ast F(x,u)\Big)f(x,u) \,\,\, \mbox{in} \,\,\, \mathbb{R}^{2},
$$
since $V,  I_\mu\ast F(x,u)  \in L^{\infty}(\mathbb{R}^{2})$ and $f(x,u) \in L^{q}(\mathbb{R}^{2})$ for $q$ large enough, we deduce by the bootstrap arguments that $u \in L^{\infty}(\mathbb{R}^{2})$ and
$$
|u(x)| \to 0 \,\,\, \mbox{as} \,\, |x| \to +\infty.
$$

We would like to point out that with few modifications, it is possible to prove the existence of solution for $(SNE)$ for other classes of potentials $V$, such as: \\

\noindent {\bf 1- First Case:} ~ $V(x)=V(|x|)$ and $f(x,s)=f(|x|,s)$ for all $x \in \mathbb{R}^{2}$ and $s \in \mathbb{R}$. \\

\noindent {\bf 2- Second Case:} ~ For all $M>0$, we assume that
$$
|\{x \in \mathbb{R}^{2}\,:\, V(x)<M\}|< +\infty,
$$
where $|\,\,\,\,|$ stands for the Lebesgue measure in $\mathbb{R}^{2}$.

\end{document}